\documentclass[a4paper,11pt]{article}
\usepackage{latexsym}
\usepackage{amssymb}
\usepackage{enumerate}
\usepackage{commath}
\usepackage{multicol}
\usepackage{comment}
\usepackage{amsfonts}
\usepackage[labelfont=bf]{caption}
\usepackage{amsmath,amsthm}
\usepackage{cases}
\usepackage[export]{adjustbox}[2011/08/13]
\usepackage{graphicx}
\usepackage[paper=a4paper,left=30mm,right=20mm,top=25mm,bottom=30mm]{geometry}

\newenvironment{@abssec}[1]{%
    \if@twocolumn

      \section*{#1}%
    \else

      \vspace{.05in}\footnotesize
      \parindent .2in
 {\upshape\bfseries #1. }\ignorespaces
    \fi}

    {\if@twocolumn\else\par\vspace{.1in}\fi}

\newenvironment{keywords}{\begin{@abssec}{\keywordsname}}{\end{@abssec}}

\newenvironment{AMS}{\begin{@abssec}{\AMSname}}{\end{@abssec}}

\newcommand\keywordsname{Key words}
\newcommand\AMSname{AMS subject classifications}
\newcommand\AMname{AMS subject classification}
\newcommand\restr[2]{{% we make the whole thing an ordinary symbol
\left.\kern-\nulldelimiterspace % automatically resize the bar with \right
#1 % the function
\vphantom{|} % pretend it's a little taller at normal size
\right|_{#2} % this is the delimiter
}}
\newtheorem{theorem}{Theorem}[section]
 \newtheorem{lemma}[theorem]{Lemma}
 
 \newtheorem{proposition}[theorem]{Proposition}
\newtheorem{remark}[theorem]{Remark}

\newcommand{\NN}{\mathbb{N}}

\newcommand{\RR}{\mathbb{R}}

\renewcommand{\SS}{\mathbb{S}}

\newcommand{\id}{{{\rm Id}}}

\def\XXint#1#2#3{{\setbox0=\hbox{$#1{#2#3}{\int}$}
\vcenter{\hbox{$#2#3$}}\kern-.5\wd0}}

\newcommand{\link}{\mathop{\circ\kern-.35em -}}

\newcommand{\ol}{\overline}
\newcommand{\pa}{\partial}

\newcommand{\dv}{{\mathop{\mathrm{div}}}}

\newcommand{\gr}{\nabla}
\newcommand{\ali}{\infty}

\newcommand{\al}{\alpha}
\newcommand{\be}{\beta}

\newcommand{\De}{\Delta}
\newcommand{\ve}{\varepsilon}
 
\newcommand{\la}{\lambda}

\newcommand{\te}{\theta}

\newcommand{\om}{\omega}
\newcommand{\Om}{\Omega}
\newcommand{\rn}{{\mathbb{R}}^N}
\newcommand{\sg}{\sigma}

\newcommand{\A}{\mathcal{A}}
\newcommand{\cA}{\mathcal{A}}

\newcommand{\C}{\mathcal{C}}
\newcommand{\cC}{\mathcal{C}}

\newcommand{\cdottone}{{\boldsymbol{\cdot}}}
\newcommand{\ri}{{\rm in}}
\newcommand{\ro}{{\rm out}}

\title{Stability analysis of the two-phase torsional rigidity \\near a radial configuration
\thanks{This research was partially supported by the Grants-in-Aid
for Scientific Research (B) ($\sharp$ 26287020)  and Challenging Exploratory Research ($\sharp$ 16K13768) of
Japan Society for the Promotion of Science.}}

\author{Lorenzo Cavallina\thanks{Research Center for Pure and Applied Mathematics,
Graduate School of Information Sciences, Tohoku
University, Sendai, 980-8579, Japan ({\tt cava@ims.is.tohoku.ac.jp}).} }

\date{}

\begin{document}

\maketitle

\vspace*{-0.5cm}
\begin{abstract}
Let $\Om_0$ denote the unit ball of $\RR^N$  ($N\ge 2$) centered at the origin. We suppose that $\Om_0$ contains a core, given by a smaller concentric ball $D_0$, made of a (possibly) different material. We discover that, depending on the relative hardness of the two materials, this radial configuration can either be a local maximizer for the torsional rigidity functional $E$ or a saddle shape. %We analyze the local optimality of this radially symmetric configuration with respect to the two-phase torsional rigidity functional $E$ under the volume and barycenter-preserving constraint. %We shall pursue the study of the second order shape derivative of the functional $E$ that originated in \cite{cava}, where perturbations acting on the core were considered. 
In this paper we consider perturbations that simultaneously act on the boundaries $\pa D_0$ and $\pa\Om_0$. This gives rise to resonance effects that are not present when $\pa D_0$ or $\pa\Om_0$ are perturbed in isolation. A detailed analysis of the sign of the second order shape derivative of $E$ is then made possible by employing the use of spherical harmonics.
\end{abstract}

\begin{keywords}
torsion problem, optimization problem, elliptic PDE, shape derivative, spherical harmonics
\end{keywords}

\begin{AMS}
49Q10
\end{AMS}

\pagestyle{plain}
\thispagestyle{plain}

\section{Introduction and main results}
\label{intro}
Let $(D,\Omega)$ be a pair of smooth bounded domains of $\rn$ ($N\ge 2$) such that $\overline{D}\subset \Omega$. The symbol $n$ will denote the outward unit normal to both $D$ and $\Omega$ and $\partial_n:=\frac{\partial}{\partial_n}$ will stand for the usual normal derivative.  
For positive $\sg$, %(here the subscripts stand for \emph{core} and \emph{shell} respectively) and 
define the following piecewise constant function:
%$$        BELLA LUNGA E MAESTOSA
%\sg_{D,\Omega}(x):=\begin{cases}
%\sg \quad &\text{ if } x\in \overline{D},\\
%1	\quad &\text{ if } x\in\Omega\setminus D.
%\end{cases}
%$$
$
\sg_{D,\Om}:= \sg \chi_{{D}}+ \chi_{\Om\setminus D},
$
(where $\chi_\cdottone$ denotes the characteristic function). 
We consider the following functional
\begin{equation}\label{E}
E(D,\Omega):=\int_{\Om} \sg_{D,\Omega} \abs{\gr u}^2,
\end{equation}
where the function $u$ is the solution (in the distributional sense) of the following two-phase bondary value problem. 
\begin{equation}\label{pb}
\begin{cases}
-\dv(\sg_{D,\Omega} \gr u)=1 \quad &\text{ in }\Omega, \\
u=0\quad &\text{ on } \partial \Omega.
\end{cases}
\end{equation}

Physically speaking, the value $E(D,\Omega)$ represents the torsional rigidity of an infinitely long beam $\Omega\times\RR$ made of two different materials (whose hardness is represented by the values $\sg$ and 1) such that their distribution in each cross sections $\Om\times\{x_{N+1}\}$ is given by the function $\sg_{D,\Om}$ for all $x_{N+1}\in\RR$. %Thus, we think that a deeper mathematical understanding of optimal shapes for the functional $E$ could be applied to the actual construction of more resistant reinforced/coated beams.

The case $D=\emptyset$ was first studied by P\'olya. In \cite{polya}, it was proved that the ball maximizes the functional $E(\emptyset,\boldsymbol{\cdot})$ among all Lipschitz domains of a given volume. We are going to provide a generalization of P\'olya's result for a two-phase setting. 
Fix $R\in(0,1)$ and let $\Omega_0$ and $D_0$ denote the open balls centered at the origin of radius $1$ and $R$ respectively. Every other pair of concentric balls can be obtained by translating, rescaling and properly choosing $R\in(0,1)$.
The aim of this paper is to study the local optimality of the symmetric configuration $(D_0,\Om_0)$. We will study how the torsional rigidity $E$ is affected by (possibly simultaneous) small perturbations of $D_0$ and $\Om_0$. %and  we will provide a second order approximation. 
To this end we will employ the use of shape derivatives up to the second order and spherical harmonic expansions. The results of this paper might find an application in the shape optimization of non-evenly coated compound materials or the study of the heat distribution for two-phase heat conductors in a stationary regime. Last, we remark that the methods presented here are suited for a various range of shape functionals. Among other works, we would like to cite \cite{conca} and \cite{sensitivity} for an in-depth analysis of the ground state energy of a two-phase conductor, modeled by the first Dirichlet eigenvalue for the operator $-\dv(\sg_{D,\Om}\gr\cdottone)$. %Finally we refer to \cite{ab} and \cite{cava}, where different kind of geometrical constraints involving the perimeter are considered.   

 %Along the same lines, without loss of generality, we will suppose that $\sg_S=1$. 
Now we will introduce some notation that will be used throughout the paper. The function $u$ will denote the solution to \eqref{pb} corresponding to the pair $(D_0,\Om_0)$. The following explicit expression for $u$ is well known:
\begin{equation}\label{u}
u(x)=\begin{cases}
\frac{1-R^2}{2N}+\frac{R^2-\abs{x}^2}{2N\sg}\quad &\text{ for } \abs{x}\in [0,R],\\
\frac{1-\abs{x}^2}{2N}\quad &\text{ for }\abs{x}\in (R,1].\\
\end{cases}
\end{equation}
We will also employ the following notation for Jacobian and Hessian matrix respectively:
$$
(D{v})_{ij}:=\frac{\partial v_i}{\partial x_j}, \quad (D^2 f)_{ij}= \frac{\partial^2 f}{\partial x_i\partial x_j},
$$
for all smooth vector field ${v}=(v_1,\dots, v_N)$ and real valued function $f$ defined on some open subset of $\rn$.
We will introduce some differential operators from tangential calculus that will be used in the sequel. For smooth $f$ and $v$ defined on $\partial D\cup \partial \Omega$  we set
\begin{equation}\label{diffop}
\begin{aligned}
\gr_\tau f &:= \gr \widetilde{f}-(\gr \widetilde{f}\cdot n)n \quad &\text{( tangential gradient)},   \\
{\dv}_\tau { v}&:= \dv \widetilde{{ v}}-n\cdot \left( D \widetilde{{ v}}\,n\right) &\text{ (tangential divergence)},
\end{aligned}
\end{equation}
where $\widetilde{f}$ and $\widetilde{v}$ are some smooth extensions on a neighborhood of $\pa D\cup \partial \Omega$ of $f$ and $v$ respectively. It is known that the differential operators defined in \eqref{diffop} do not depend on the choice of the extensions.   
Moreover we denote by $D_\tau {v}$ the matrix whose $i$-th row is given by $\gr_\tau v_i$.
We define the (additive) mean curvature as $H:=\dv_\tau n$ (cf. \cite{henrot, SG}); according to this definition, the mean curvature $H$ of $\pa D_0$ is given by  $(N-1)/R$.   
Finally, for any sufficiently smooth function $f$ defined in a neighborhood of $\pa D$, its {\em jump through the interface} $\pa D$ will be denoted by $[f]:= f_+-f_-$, where $f_+$ and $f_-$ are the traces of $f$ on $\pa D$ taken from the outside and inside respectively.  
The following theorem involving the first order shape derivative of $E$ will be proved in Subection \ref{ssec 3.2}.
\newtheorem{test}{Theorem}
\renewcommand\thetest{\Roman{test}}
\begin{test}\label{thm I}
For smooth perturbations that fix the volume (at least at first-order), the first shape derivative of $E$ at $(D_0,\Om_0)$ vanishes.
\end{test} 
The following result is an improvement of Theorem \ref{thm I} and will be proved in the end of Subsection \ref{ssec 4.3}. 
\begin{test}\label{thm II}
If $\sg\ge1$, $(D_0,\Om_0)$ is a local maximum for the functional $E$ under the volume and barycenter-preserving constraint. Moreover, if $\sg<1$, $(D_0,\Om_0)$ is a saddle shape for $E$ under the above-mentioned constraint. 
\end{test}
 
In Section 2, known facts about shape derivatives will be presented. Moreover, perturbations subject to various constraints will be introduced. Section 3 will be devoted to the computation of the first order shape derivatives of both the state function $u$ and the shape functional $E$. We will also give a proof of Theorem \ref{thm I} here. In Section 4, we will deal with the computation of the second order shape derivative of $E$. Here, the spherical harmonic expansion of $u'$, performed in Subsection \ref{ssec 3.1}, will play a crucial role in determining the sign of $E''$.  
\section{Preliminaries on shape derivatives}\label{sec 2}
\setcounter{equation}{0}
%In this section we will introduce the necessary background from the theory of domain perturbation and shape derivatives that will be used in the sequel.
\subsection{Basic definitions and structure formula}\label{ssec 2.1}
We are interested in the following class of smooth perturbations:
$$
\mathcal{A}:=\Big\{ \Phi \in \C^\infty\big([0,1)\times \rn,\rn\big)  \;\Big|\;  \Phi(0,\boldsymbol{\cdot})\equiv 0     \Big\}.
$$

For $\Phi \in \mathcal{A}$, $t\in[0,1)$ and an arbitrary domain $\omega\subset\rn$ we set 
%\begin{equation}\nonumber
%\begin{aligned}
$\Phi(t):=\Phi(t,\boldsymbol{\cdot})$ and $\left({\rm Id}+\Phi(t)\right)\omega:=\{x+\Phi(t,x)\;\vert\; x\in \omega\}$.
%\end{aligned}
%\end{equation}
When no confusion arises, we will also write $\omega_t$ for $\left({\rm Id}+\Phi(t)\right)\omega$.   
By assumption, for a given $\Phi\in\cA$ there exists some smooth vector field $h$ such that the following expansion holds:
\begin{equation}\label{whosh}
\Phi(t)= t h +o(t) \quad \text{as }t\to0.
\end{equation}
%It will be useful to decompose $h$ into its normal and tangential components as follows: 
%$$
%h_n:=h\cdot n, \quad  h_\tau : = h- h_n\, n \text{ on }\partial \om.
%$$
%Notice that the normal component of $h$ is a scalar quantity, while $h_\tau$ is a vector.

For any shape functional $J$, domain $\omega$ and deformation field $\Phi\in\cA$, we define the \emph{shape derivative of $J$ {with respect to} $\Phi$ at $\omega$} as the following quantity
\begin{equation}\label{def shape der}
J'(\Phi):=\restr{\frac{d}{dt} J\Big(\left({\rm Id}+\Phi(t)\right)\omega\Big)}{t=0} = \lim_{t\to 0}\, \frac{J\Big(\left({\rm Id}+\Phi(t)\right)\omega\Big)-J(\omega)}{t}.
\end{equation}
Second order shape derivatives are defined analogously (we refer to \cite{simon} for one of the first general works on the topic and \cite{new} for a more recent developement of the theory in the framework of differential forms). We note that, when dealing with a functional that takes several domains as input (like the functional $E$, defined in \eqref{E}, does), definition \eqref{def shape der} has to be modified accordingly, applying ${\rm Id}+\Phi(t)$ to each domain. 

Usually, shape functionals depend on the input domain by means of a real valued function, called {\em state function} in the literature.  
Here we give the definition of both {\em shape derivative} and {\em material derivative} of a state function. 
Fix an admissible perturbation $\Phi\in\A$ and let $u=u(t,x)\in \cC^1\big( [0,1), \cC^1(\Om_t,\RR)\big)$.
The {\em shape derivative} of the state function $u$ is defined as the following partial derivative with respect to $t$ at a fixed point $x\in\Om$:
$$
u'(t_0,x):= \frac{\partial u}{\partial t} (t_0,x), \;\text{ for }x\in\Omega, t_0\in [0,1). 
$$ 
On the other hand, differentiating along the trajectories gives rise to the {\em material derivative}:
$$
\dot{u}(t_0,x):= \frac{\partial v}{\partial t}(t_0,x), \;x\in\Omega, t_0\in [0,1),
$$ 
where $v(t,x):=u(t, x+\Phi(t,x))$.
From now on for the sake of brevity we will omit the dependency on the ``time" variable unless strictly necessary and write $u(x)$, $u'(x)$ and $\dot{u}(x)$ for $u(0,x)$, $u'(0,x)$ and $\dot{u}(0,x)$. 
The following relationship between shape and material derivatives hold true:
\begin{equation}
\label{u'a} u'=\dot{u}-\gr u \cdot h.
%u''&= \ddot u -2\gr u'\cdot \h - \gr u\cdot (D\h \, \h)- D^2 u (\h,\h),
\end{equation}
%where $D$ and $D^2$ denote the Jacobian and Hessian matrix respectively.

We are interested in the case where $u(t,\boldsymbol{\cdot}):=u_{t}$ is the solution to problem \eqref{pb} corresponding to the distribution $\sg_t=\sg_{D_t,\Om_t}$: we will agree on the fact that the function $u$ defined in \eqref{u} corresponds to $u_0$ and that $\sg_0=\sg_{D_0,\Om_0}$.
%In this case, since by symmetry we have $\gr u = (\dn u) n$, the formula above admits the following simpler form on $\partial D_0\cup \partial\Omega_0$:
%\begin{equation}\label{u'}
%u'=\dot{u}-(\dn u)\, h\cdot n.
%\label{u''} u''&= \ddot u -2(\dn u')\hn -(\dn u) \n\cdot (D\h) \n\, \hn- (\partial_{nn} u) h_n^2.
%\end{equation}

We will now state a very important result, namely the following {\em structure theorem}  for shape derivatives (cf. \cite[Theorem 5.9.2, p. 220]{henrot} and the subsequent corollaries). 
For every shape functional $J$, domain $\om$ and pertubation field $\Phi$ in $\mathcal{A}$, under suitable smoothness assumptions the following holds:
\begin{equation}\label{expansion}
J(\omega_t)=J(\omega)+t \, l_1^J (h\cdot n)+\frac{t^2}{2}\left(  l_2^J(h\cdot n,h\cdot n)+l_1^J(Z)  	\right) + o(t^2) \; \text{ as }t\to 0,
\end{equation}
for some linear $l_1^J: \C^\infty (\partial \omega)\to \RR$ and bilinear form $l_2^J: \C^\infty (\partial \omega)\times \C^\infty (\partial \omega)\to\RR$ to be determined eventually. The term $l_1^J(Z)$ in \eqref{expansion} corresponds to an ``acceleration'' term due to the tangential component of the perturbation field $\Phi$. Namely, by \cite[Corollary 5.9.3, p. 221]{henrot}, we have
\begin{equation}\label{Z}
Z:=\left( V'+(D h) h   \right)\cdot n+ ((D_\tau n) h_\tau)\cdot h_\tau-2  h_\tau \cdot \gr_\tau (h\cdot n),
\end{equation}
where $V(t,\Phi(t)):=\partial_t \Phi(t)$ and $V':=\partial_t V(t,\boldsymbol{\cdot})$ and $h_\tau := h-(h\cdot n)n$ is the tangential component of the vector field $h$.
%$$Z:=((D_\tau\n) \h_\tau)\cdot \h_\tau-2\gr_\tau\hn\cdot \h_\tau.$$

We are going to apply the expansion \eqref{expansion} to the functional $E$ at the configuration given by $(D_0,\Om_0)$. The linear form $l_1^E$ will be computed in Subsection \ref{ssec 3.2}, while the computation of the bilinear form $l_2^E$ will be the topic of Subsection \ref{ssec 4.1}. There will be no need to compute the function $Z$ (defined by \eqref{Z}) directly. Its computation will be avoided by employing the $2^{\rm nd}$ order volume-preserving condition \eqref{2nd} that will be derived in the next subsection.

\subsection{Perturbations verifying some geometrical constraints}\label{ssec 2.2}

In this paper we are dealing with a constrained optimization problem (we refer to \cite{bandle} for a slightly different, though equivalent way of analyzing the second order shape derivative of functionals subject to geometrical constraints). 
Let $|\cdottone|$ denote the volume (Lebesgue measure) of a set and define the class of perturbations in $\A$ that perturb  $D_0$ and $\Omega_0$ while keeping their volume fixed as:
$$
\cA^*:= \Big\{  \Phi\in \mathcal{A}  \;\Big |\;  \abs{(\id+\Phi(t))D_0}=\abs{D_0} \text{ and }  \abs{(\id+\Phi(t))\Omega_0}=\abs{\Omega_0} \text{ for all } t\in [0,1)  \Big\}.
$$
%We will simply write $\B$ in place of $\B(\br)$.
The following expansion of the volume functional $\abs{\boldsymbol{\cdot}}$ is well knwon (it is a consequence of the structure formula \eqref{expansion}, the Hadamard formula \cite[Corollary 5.2.8, p. 176]{henrot} and \cite[Example 1, p. 225]{henrot}):
for all bounded domains $\om\subset\rn$ and $\Phi\in\cA$ we have
\begin{equation}\label{volex}
\abs{\omega_t}=\abs{\omega}+ t \int_{\partial\omega} h\cdot n + \frac{t^2}{2} \left( \int_{\partial \omega} H (h\cdot n)^2 + \int_{\partial\omega} Z     \right) + o(t^2) \text{ as }t\to0.
\end{equation}
In particular, for all $\Phi\in\cA^*$, this yields the following two conditions for $\omega=D_0$, $\Omega_0$: 
\begin{align}
&\int_{\partial \omega} h\cdot n =0,  \quad\quad\quad&\text{($1^{\rm st}$ order volume preserving)} \label{1st}\\
&\int_{\partial \omega} H (h\cdot n)^2+ \int_{\partial \omega}Z=0. &\text{($2^{\rm nd}$ order volume preserving)}\label{2nd}
\end{align} 

Moreover, for every perturbation $\Phi\in\cA$, it will be useful for our purposes to separate its contributions on $\pa D_0$ and $\pa\Om_0$.
Take some positive constants $R_1$ and $R_2$ such that $R<R_1<R_2<1$ and define
\begin{equation}\nonumber\label{Ainout}
%\begin{aligned}
\cA_{\rm in}:= \big\{\Phi\in\cA \; \big|\; \Phi(t,x)=0 \quad \text{if }\abs{x}\ge R_2 \big\}, \quad 
%\end{aligned}
\cA_{\rm out}:= \big\{\Phi\in\cA \; \big|\; \Phi(t,x)=0 \quad \text{if }\abs{x}\le R_1 \big\}. 
\end{equation}
Notice that for every $\Phi\in\cA$ there exist some $\Phi_{\rm in}\in\cA_{\rm in}$ and $\Phi_{\rm out}\in\cA_{\rm out}$ such that 
$
\Phi=\Phi_{\rm in}+\Phi_{\rm out},
$
 moreover the values of $\Phi_{\rm in}$ and $\Phi_{\rm out}$ are uniquely determined (and actually equal to $\Phi$) on $\ol{B_{R_1}}$ and $\rn\setminus B_{R_2}$ respectively. We will set $\Phi_\ri= t h_\ri + o(t)$ and $\Phi_\ro = t h_\ro + o(t)$ for $t\to 0$.  
In a similar manner we put:
$$
\cA_{\rm in}^*:=\cA^*\cap \cA_{\rm in}, \quad \quad \cA_{\rm out}^*:= \cA^*\cap\cA_{\rm out}.
$$
Finally, we recall that the functional $E$ is translation invariant, namely $E(D+x_0,\Om+x_0)=E(D,\Om)$ for any $x_0\in\RR^N$. Therefore, the following class of perturbation fields that fix both the {\bf volume} of the domains $D_t$ and $\Om_t$, and the {\bf barycenter} of $\Om_t$ is the most natural for our purposes:  
$$
\A_{\rm bar}^* :=\left\{\Phi\in\A^* \;\middle|\;  \int_{\Om_t} x dx=0\right \}
$$
By an application of the Hadamard formula (see \cite[Corollary 5.2.8, p. 176]{henrot}) we see that, for all $\A_{\rm bar}^*\ni \Phi=\Phi_\ri+\Phi_\ro$ we must have 
\begin{equation}\label{barycenter preserving}
\int_{\pa\Om_0}x_i h_\ro\cdot n=0 \quad \mbox{ for all } i=1,\dots,N.
\end{equation}
We conclude this section by citing an extension result for volume-preserving perturbations (see \cite[Remark 2.2]{cava} for an explicit construction)
\begin{remark}\label{rmk extension}
For any smooth $h_{\rm in}:\pa D_0\to \rn$ that satisfies the first order volume-preserving condition \eqref{1st}, there exists a perturbation field $\Phi_{\rm in}\in\cA_{\rm in}^*$ such that $\Phi_{\rm in}(t)=t h_{\rm in} + o (t) $ as $t\to 0$. Analogous results hold for any smooth function $h_{\rm out}:\pa\Om_0\to\rn$ that satisfies \eqref{1st} on $\pa \Om_0$ or \eqref{barycenter preserving}.
\end{remark}

\section{$1^{\rm st}$ order shape derivatives} \label{sec 3}
\setcounter{equation}{0}
In this section we will compute the first order derivative of the state function $u$ and that of the functional $E$. As predicted by the general theory, $u'$ will not appear in the final expression of $E'$. By combining these results with the first order volume-preserving condition \eqref{1st} we will provide an immediate proof of Theorem \ref{thm I}.
\subsection{Shape derivative of the state function}\label{ssec 3.1}
\begin{proposition}\label{prop u'}
Take an arbitrary perturbation field $\Phi\in\A$ and set: $\Phi=\Phi_{\rm in}+\Phi_{\rm out}$ for some $\Phi_{\rm in}\in\A_{\rm in}$ and $\Phi_{\rm out}\in\A_{\rm out}$. The the first order shape derivative $u'$ of the state function $u$, computed with respect to the pertubation field $\Phi\in\A$, can be decomposed as $u'=u'_{\rm in}+u'_{\rm out}$, where $u'_{\rm in}$ and $u'_{\rm out}$ are the solutions to the following.
\noindent\begin{minipage}{.5\linewidth}
\begin{equation}\label{u'in}
\begin{cases}\De u'_{\rm in}=0  \quad\mbox{ in } D_0\cup (\Om_0\setminus\ol{D_0}),\\
[\sg_0\, \pa_n u'_{\rm in}]=0 \quad\mbox{ on }\pa D_0,\\
[u'_{\rm in}]=-[\pa_n u] h_{\rm in}\cdot n\quad\mbox{ on }\pa D_0,\\
u'_{\rm in}=0 \quad \mbox{ on }\pa\Om_0.
\end{cases}
\end{equation}
\end{minipage}%
\begin{minipage}{.5\linewidth}
\begin{equation}\label{u'out}
\begin{cases}
\De u'_{\rm out}=0  \quad \mbox{ in } D_0\cup (\Om_0\setminus\ol{D_0}),\\
[\sg_0\, \pa_n u'_{\rm out}]=0 \quad\mbox{ on }\pa D_0,\\
[u'_{\rm out}]=0\quad \mbox{ on }\pa D_0,\\
u'_{\rm out}= -\pa_n u\, h_{\rm out}\cdot n \quad \mbox{ on }\pa\Om_0.
\end{cases}
\end{equation}
\end{minipage}
\end{proposition}
\begin{proof}
By linearity, we get $u'=u'_\ri+u'_\ro$, where $u'_\ri$ and $u'_\ro$ are the shape derivatives of the state function $u$ computed with respect to the perturbation fields $\Phi_\ri$ and $\Phi_\ro$ respectively. 
The computation of $\Phi_\ri$ has been carried out in \cite[Proposition 2.3]{cava} and therefore we will refer to it for a proof of \eqref{u'in}.
Now we will give a proof of \eqref{u'out}. We compute the shape derivative $u'_\ro$ of $u$ with respect to $\Phi_\ro$. To this end, first we formally differentiate the equation $-\dv(\sg_t \gr u_t)=1$ in $\Om_t$ with respect to the variable $t$
(here $\sg_t$, $u_t$ and $\Om_t$ are defined according to the perturbation field $\Phi_\ro(t)$).
We get $\dv (\sg_0 \gr u'_\ro)=0$ in $\Om_0$, which is equivalent to the first three equalities in \eqref{u'out}. This formal calculation can be justified rigorously as done in \cite[Proposition 2.3]{cava} (nevertheless, in this case we do not need to split the domain $\Om_t$ as $D_0\cup\Om_t\setminus\ol{D_0}$ because $\Phi_\ro$ vanishes on a neighborhood of the interface $\pa D_0$). Finally, the boundary condition in \eqref{u'out} is derived by differentiating the constant Dirichlet boundary condition on $\partial \Omega_t$ by means of \eqref{u'a}. 
\end{proof}

In the final part of this subsection we will present an explicit way to express the solutions to problems \eqref{u'in} and \eqref{u'out} that will be used in Subsections \ref{ssec 4.2} and \ref{ssec 4.3}. We will employ the use of spherical harmonics $Y_{k,i}$, i.e. the solutions to the eigenvalue problem
%\begin{equation}\nonumber
$-\De_\tau Y_{k,i}:=-\dv_\tau(\gr_\tau Y_{k,i})= \lambda_{k} Y_{k,i}$ on $\SS^{N-1}$, for $k\in\{0,1,\dots\}$ and $i\in\{1,\dots,d_k\}$,
%\end{equation}
where %$\De_\tau(\cdottone):=\dv_\tau (\gr_\tau \cdottone)$ denotes the Laplace--Beltrami operator on the unit sphere $\SS^{N-1}$, and 
$d_k$ is the multiplicity of the $k$-th eigenvalue $\la_k=k(N+k-2)$. We will also impose the normalization $\norm{Y_{k,i}}_{L^2(\SS^{N-1})}=1$, so that the family $(Y_{k,i})_{k,i}$ becomes a complete orthonormal system of $L^2(\SS^{N-1})$.
\begin{proposition}\label{prop u' kepeken sphar}
Take $\Phi\in\A^*$. Using the same notation of Proposition \ref{prop u'}, suppose that, for some real constants $\al_{k,i}^\ri$ and $\al_{k,i}^\ro$ (the indexes $k$ and $i$ ranging as before), the following expansions hold  for all $\te\in\SS^{N-1}$
\begin{equation}\label{h_in h_out exp}
 (h_{\ri}\cdot n)(R\te)=\sum_{k=1}^\infty \sum_{i=1}^{d_k} \al_{k,i}^\ri \,Y_{k,i}(\te),	\quad
( h_{\ro}\cdot n)(\te)=\sum_{k=1}^\infty \sum_{i=1}^{d_k} \al_{k,i}^\ro\, Y_{k,i}(\te).
\end{equation}
Then, the shape dervatives $u'_\ri$ and $u'_\ro$ admit the following explicit expression for $\te\in\SS^{N-1}$ and $\ve\in\{\ri,\ro\}$:
$$
u'_\ve(r\te)=\begin{cases}
\displaystyle \sum_{k=1}^\infty\sum_{i=1}^{d_k}\al_{k,i}^\ve \, B_k^\ve r^k\, Y_{k,i}(\te) \quad &\mbox{ for } r\in[0,r],\\
\displaystyle \sum_{k=1}^\infty\sum_{i=1}^{d_k}\al_{k,i}^\ve \left( C_k^\ve r^{2-N-k} + D_k^\ve r^k \right) Y_{k,i} (\te) \quad &\mbox{ for } r\in(R,1],
\end{cases}
$$
where the constants $B_k^\ve$, $C_k^\ve$ and $D_k^\ve$ are defined as follows:
%%%% EQUATION COMMENTED OUT
%\begin{comment}
%\begin{equation}
%\begin{aligned}
%B_k^\ri = -  \frac{R^{-k+1}(\sg-1)}{\sg N}  \frac{(M+k) P+ k}{P(k \sg+M+k)+k(1- %\sg)},\\
%C_k^\ri =    \frac{R^{-k+1}(\sg-1)}{N}  \frac{ k }{P(k \sg+M+k)+k(1-\sg)}, \\
%D_k^\ri =  \frac{R^{-k+1}(1-\sg)}{N}  \frac{ k }{P(k \sg+M+k)+k(1-\sg)}. \\ 
%\end{aligned}
%\begin{aligned}
%B_k^\ro =    \frac{-1}{N}\frac{(2-N-2k) R^{2-N-2k}}{(k  {\sg}-2+N+k ) P+k(1-\sg) %},\\ 
%C_k^\ro =  \frac{1}{N}  \frac{k(1-\sg)}{(k {\sg}-2+N+k  ) P+ k(1-\sg) },\\ 
%D_k^\ro =   \frac{1}{N} \frac{(k {\sg}-2+N+k ) P}{(k  {\sg}-2+N+k  ) P+k(1-\sg) }.
%\end{aligned}
%\end{equation}
%%%END OF COMMENT
\begin{equation*}
\begin{aligned}
&B_k^\ri = (1-\sg)R^{-k+1}\left( (N-2+k)R^{2-N-2k}\right)/F, \quad\quad  
C_k^\ri = -D_k^\ri = (\sg-1)k R^{-k+1}/F, \\
&B_k^\ro = (N-2+2k)R^{2-N-2k}/F, \quad C_k^\ro = (1-\sg)k/F,\quad D_k^\ro = (N-2+k+k\sg)R^{2-N-2k}/F,
\end{aligned}
\end{equation*}
and the common denominator $F= N(N-2+k+k\sg)R^{2-N-2k}+k N (1-\sg)>0$.
\end{proposition}
For the details of the proof we refer to \cite[Section 4]{cava}, where the explicit expression of $u'_\ri$ has been derived from \eqref{u'in} by separation of variables. Since the derivation of $u'_\ro$ is completely analogous, we will omit the proof alltogether.  
\subsection{Computation of $E'$}\label{ssec 3.2}
\begin{theorem}\label{l1}
For all $\Phi\in\A$ we have
\begin{equation}\label{E'}
E'(\Phi)= E'_\ri(\Phi)+E'_\ro(\Phi) = -\int_{\pa D_0} [\sg_0 |\gr u|^2]\,h_\ri \cdot n + \int_{\pa\Om} |\gr u|^2\, h_\ro \cdot n.
\end{equation}
In particular, if $\Phi$ satisfies \eqref{1st} on  both $\pa D_0$ and $\pa\Om_0$, then $E'(\Phi)=0$.
\end{theorem}
\begin{proof}
First, we decompose $\Phi$ as $\Phi_\ri+\Phi_\ro$ for some $\Phi_\ri\in\A_\ri$ and $\Phi_\ro\in\A_\ro$. By the general theory we know that $E'$ can be written as a linear form $l_1^E$ of $h\cdot n$, thus
$$
E'(\Phi)=l_1^E(h\cdot n)= l_1^E(h_\ri\cdot n)+l_1^E(h_\ro\cdot n).
$$
We refer to \cite[Theorem 2.4]{cava}, where the linear form $l_1^E(h_\ri\cdot n)$ has been computed in detail: this yields the first summand of \eqref{E'}. 

We will now compute $l_1^E(h_\ro\cdot n)$, the shape derivative of $E$ with respect to $\Phi_\ro$. 
To this end we apply the Hadamard formula (see \cite[Corollary 5.2.8, p. 176]{henrot}) to $E(D_0,\Omega_t)=\int_{\Omega_t} \sg_{D_0,\Omega_t} \abs{\gr u_t}^2$. We get
$$
l_1^E(h_\ro \cdot n)=2\int_\Om \sg_0 \gr u \cdot \gr u'_\ro + \int_{\pa\Om_0} \abs{\partial_n u}^2\,h_\ro \cdot n.
$$
Finally, integrating the first equation of \eqref{u'out} against $u$ yields $\int_\Om \sg_0 \gr u'_\ro\cdot \gr u=0$ and thus \eqref{E'} is proved. 

Now, as $|\gr u|$ is constant on both $\pa D_0$ and $\pa \Om_0$, it is immediate to see that $E'(\Phi)$ vanishes when both $h_\ri$ and $h_\ro$ verify \eqref{1st} (thus, in particular, for all $\Phi\in\A^*$).
\end{proof}

\section{$2^{\rm nd}$ order shape derivative} \label{sec 4}\setcounter{equation}{0}
%In this section we will compute the second order shape derivative $E''$ of $E$ and analyze its sign.  
\subsection{Computation of $E''$}\label{ssec 4.1}
The computation of $E''(\Phi)=l_2^E(h\cdot n,h\cdot n)+ l_1^E(Z)$ for an arbitrary $\Phi\in\A^*$ will require two steps. First, we will compute the bilinear form $l_2^E(h\cdot n, h\cdot n)$ with the aid of special perturbations (called {\em Hadamard perturbations} in the literature) and finally we will take care of the term involving $Z$, using the second order volume-preserving condition \eqref{2nd}.  
\begin{proposition}\label{prop l2e}
Let $\Phi\in\A$. Then, the term $l_2^E$ that appears in the expansion \eqref{expansion} admits the following explicit expression:
\begin{equation}\nonumber
\begin{aligned}
&l_2^E(h\cdot  n, h\cdot n)=2\int_{\pa \Om_0}  \gr u \cdot \gr u' \,(h\cdot n) + 2\int_{\pa \Om_0}  \pa_n u \,\pa_{nn} u (h\cdot n)^2 + \int_{\pa \Om_0} |\gr u|^2 H (h\cdot n)^2 \\
& -2\int_{\pa D_0} \left[ \sg_0 \gr u \cdot \gr u'\right] (h\cdot n) - 2\int_{\pa D_0}  \sg\pa_n u_- [\pa_{nn} u] (h\cdot n)^2 - \int_{\pa D_0} \left[ \sg_0 |\gr u|^2\right] H (h\cdot n)^2.
\end{aligned}
\end{equation}
\end{proposition}
\begin{proof}
We will follow the steps of \cite[Subsection 5.9.6, pp. 226--227]{henrot}. We know, by \eqref{expansion}, that $E''(\Phi)$ can be written as the sum of a quadratic form $l_2^E$ of $h\cdot n$ and a linear form $l_1^E$ of the function $Z$ defined in \eqref{Z}. %Since the linear form $l_1^E$ has already been dealt with in Theorem \ref{l1}, we just need to compute $l_2^E$.
Moreover (see  \cite[Corollary 5.9.4, p. 221]{henrot} or \cite{structure}) it is woth noticing that $Z$ vanishes in the special case when 
\begin{equation}\label{hadamard pert}
\Phi(t)=t\,(h\cdot n )n \quad \mbox{ on }\pa D_0\cup \pa\Om_0.
\end{equation}
In other words, $E''(\Phi)=l_2^E(h\cdot n,h\cdot n)$ for all perturbations $\Phi\in\A$ satisfying \eqref{hadamard pert}.
Therefore, for all $\Phi\in\A$ satisfying \eqref{hadamard pert}, by employing the explicit form of the first order shape derivative given by Theorem \ref{l1}, we can write
\begin{equation}\nonumber
\begin{aligned}
l_2^E(h\cdot n, h\cdot n)=
 \restr{\frac{d}{dt}}{t=0}\int_{\pa D_t} \left[\sg |\gr u_t|^2\right]\left\{ h \circ \left(\id+\Phi(t)\right)^{-1}\right\}\cdot n_t \\
 +  \restr{\frac{d}{dt}}{t=0} \int_{\pa \Om_t} \sg |\gr u_t|^2\left\{ h \circ \left(\id+\Phi(t)\right)^{-1}\right\}\cdot n_t,
\end{aligned}
\end{equation}
where $n_t$ denotes the outward unit normal to both $\pa D_t$ and $\pa \Om_t$. 
The formula above can be rewritten in the following compact way:
\begin{equation}\label{(A) and (B)}
l_2^E(h\cdot n, h\cdot n)= \underbrace{\restr{\frac{d}{dt}}{t=0} \int_{\pa D_t} \!\!\! f(t)\, \left\{h\circ \left( \id+\Phi(t)\right)^{-1}\right\}\cdot n_t^1}_{(A)} + \underbrace{\restr{\frac{d}{dt}}{t=0} \int_{\pa (\Om_t\setminus \ol{D_t})}\!\!\! f(t)\, \left\{h\circ \left( \id+\Phi(t)\right)^{-1}\right\}\cdot n_t^2}_{(B)},
\end{equation}
where $f(t):=\sg_t |\gr u_t|^2$, and $n_t^1$ (respectively $n_t^2$) denotes the unit normal vector to $\pa D_t$ (respectively $\pa (\Om_t\setminus \ol{D_t})$) pointing in the outward direction with respect to the domain $D_t$ (respectively $\Om_t\setminus \ol{D_t}$). Moreover, the surface integrals in the above have to be intended in the sense of traces, taken from the inside of the respective domains.
We first deal with the term $(A)$ of \eqref{(A) and (B)}. We get %By the divergence theorem, $(A)$ can be rewritten as 
$$
(A)= \restr{\frac{d}{dt}}{t=0} \int_{D_t} \dv\left( f(t)\, h\circ \left(\id+\Phi(t)\right)^{-1}\right),
$$
The divergence theorem, followed by an application of the Hadamard formula (see \cite[Corollary 5.2.8, p. 176]{henrot}), yields 
$$
(A)= \int_{D_0} \restr{\frac{\pa}{\pa t}}{t=0} \dv\left( f(t) \, h\circ\left(\id+ \Phi(t)\right)^{-1}\right) + \int_{\pa D_0} \dv\left( f(0) h\right) \, h\cdot n=(A1)+(A2).
$$
We have
\begin{equation}\nonumber
\begin{aligned}
(A1)= %\int_{D_0} \dv\left( f' h - f(0) \, D h \, h\right) = 
\int_{\pa D_0} f' h \cdot n - \int_{\pa D_0} f(0) \left(D h \, h\right)\cdot n,%\\
\quad (A2)= \int_{\pa D_0} \left( \gr f(0)\cdot h + f(0)\dv h\right) h\cdot n.
\end{aligned} 
\end{equation}
Moreover, as $h=(h\cdot n)n$ on $\pa D_0$ by hypothesis, we get
\begin{equation}\label{(A) ii kanji}
(A)=\int_{\pa D_0} f' h\cdot n+ \int_{\pa D_0} \pa_n f(0) (h\cdot n)^2
+ \int_{\pa D_0} f(0)\left(\dv h - n\cdot (Dh \, h)\right)\, h\cdot n.
\end{equation}

Now, by the definition of tangential divergence in \eqref{diffop}, \eqref{hadamard pert} and \cite[equation (5.22) p. 366]{SG} we obtain:
$
\dv h - n\cdot (Dh \, n)= \dv_\tau h =\dv_\tau \big( (h\cdot n) n\big)= H \, h\cdot n
$.

Recalling the definition of $f(t)$, we can rewrite \eqref{(A) ii kanji} as follows
$$
(A)= 2\int_{\pa D_0}\sg \gr u_- \cdot \gr u'_- \, (h\cdot n) + 2\int_{\pa D_0} \sg \pa_n u_- (\pa_{nn} u _-) (h\cdot n)^2+\int_{\pa D_0}\sg |\gr u_-|^2 H (h\cdot n)^2.
$$
The term labelled $(B)$ in \eqref{(A) and (B)} can be computed analogously. The claim of Proposition \ref{prop l2e} is finally obtained by combining the two terms $(A)$ and $(B)$.
%\begin{equation}
%\begin{aligned}
%&l_2^E(h\cdot  n, h\cdot n)=+2\int_{\pa \Om_0}  \gr u \cdot \gr u' \,(h\cdot n) + 2\int_{\pa \Om_0}  \pa_n u \,\pa_{nn} u (h\cdot n)^2 + \int_{\pa \Om_0} |\gr u|^2 H (h\cdot n)^2 \\
%& -2\int_{\pa D_0} \left[ \sg_0 \gr u \cdot \gr u'\right] (h\cdot n) - 2\int_{\pa D_0}  \sg\pa_n u_- [\pa_{nn} u] (h\cdot n)^2 - \int_{\pa D_0} \left[ \sg_0 |\gr u|^2\right] H (h\cdot n)^2.
%\end{aligned}
%\end{equation}
\end{proof}
In light of the structure formula \eqref{expansion}, the following theorem is an immediate consequence of Proposition \ref{prop l2e} and the second order volume-preserving condition \eqref{2nd} combined.
\begin{theorem}\label{thm E''=}
For all $\Phi\in\A^*$, the following holds:    
\begin{equation}\nonumber
\begin{aligned}
E''(\Phi)=&+2\int_{\pa \Om_0}  \gr u \cdot \gr u' \,(h\cdot n) + 2\int_{\pa \Om_0}  \pa_n u \,\pa_{nn} u (h\cdot n)^2 +\\
& -2\int_{\pa D_0} \left[ \sg_0 \gr u \cdot \gr u'\right] (h\cdot n) - 2\int_{\pa D_0}  \sg\pa_n u_- [\pa_{nn} u] (h\cdot n)^2.
\end{aligned}
\end{equation}
\end{theorem}

\subsection{Analysis of the non-resonant part}\label{ssec 4.2}
In this subsection we will suppose that the expansion \eqref{h_in h_out exp} holds true for $h_\ri$ and $h_\ro$. Combining the result of Theorem \ref{thm E''=} and the explicit expressions for $u$ and $u'=u'_\ri+u'_\ro$ (given by \eqref{u} and Proposition \ref{prop u' kepeken sphar} respectively) yields the following. 
\begin{equation}\label{E''=in+out+res}%\nonumber
E''(\Phi)=\sum_{{k=1}}^\ali\sum_{{i=1}}^{d_k}\left\{     \left(\al_{k,i}^\ri\right)^2 E''_\ri (k)  +\left(\al_{k,i}^\ro\right)^2 E''_\ro (k)  + \al_{k,i}^\ri\, \al_{k,i}^\ro \, E''_{\rm res}(k)\right\},
\end{equation}
%where the symbol $\de_{\cdottone,\cdottone}$ denotes the usual Kronecker delta,
where
\begin{equation}\label{in out res li seme}%\nonumber
\begin{aligned}
E''_\ri(k)&= \frac{2 R^N}{N}  \left( \frac{1-\sg}{\sg} \right) \left(F - k\left( k(1-\sg)+(N-2+k)(1-\sg)R^{2-N-2k}\right) \right)\bigg/F,\\
E''_\ro(k) &= \frac{2}{N}\left( F-k\left( (-N+2-k)(1-\sg)+(N-2+k+k\sg)R^{2-N-2k} \right)\right)\bigg/ F,\\
E''_{\rm res}(k)&= \frac{4 (\sg-1) R^{1-k}}{N}\left((N-2)k+2k^2\right)\bigg/ F,
\end{aligned}
\end{equation}
and $F$ is the term defined at the end of the statement of Proposition \ref{prop u' kepeken sphar}. 

In this subsection we will consider only the coefficients $k\in\{1,2,\dots\}, i\in\{1,\dots d_k\}$ such that $\al_{k,i}^\ri\, \al_{k,i}^\ro =0$ (in other words we will consider only the non-resonant part of $E''$).  
%\begin{equation}\label{non resonance}
%\al_{k,i}^\ri\, \al_{k,i}^\ro =0. %\quad \mbox{ for all } k \in \{1,2,\dots\} \mbox{ and } i\in\{1,\dots,d_k\}, j\in\{1,\dots,d_m\}.
%\end{equation}
Under this assumption the contributions of $E''_\ri(k)$ and $E''_\ro(k)$ can be analyzed separately. We have the following result.
\begin{lemma}\label{lemma decreasing}
The functions $\NN\ni k\mapsto E''_\ve (k)$, for $\ve\in\{\ri,\ro\}$, defined by \eqref{in out res li seme}, are monotone decreasing. The function $E''_\ri$ is strictly decreasing for $\sg\ne1$ and constantly zero otherwise, while $E''_\ro$ is a strictly decreasing function for all $\sg>0$. 
\end{lemma}
\begin{proof}
We refer to \cite[Lemma 4.1]{cava} for the proof of the monotonicity of $E''_\ri$. Here we will deal with $E''_\ro$ only.
In the following, we will replace the integer parameter $k$ with a real variable $x$ and study the function $x\mapsto E''_\ro(x)$ in $(0,\infty)$.
The calculations are going to be pretty long, although elementary. For the sake of readability we will adopt the following notation:
\begin{equation}\label{notation pona}
L:=R^{-1}>1,\quad \lambda:=\log(L)>0,\quad M:=N-2\ge0;\quad P=P(x):=L^{2x+M}>1.
\end{equation}
Differentiating the expression for $E''_\ro$ in \eqref{in out res li seme} by $x$ yields the following expression
$$
\frac{d}{dx}E''_\ro(x)=  \frac{2\left(a(x)+\sg b(x)+\sg^2 x^2 c(x)\right)}{F^2},
$$
where we have set
\begin{equation}\nonumber
\begin{aligned}
a(x)&:=x^2 P^{-1}+ M(2x+M) - (x+M)^2 P - 2\lambda (2x^3+3 Mx^2+M^2 x),\\
b(x)&:=-2x^2 P^{-1} - M(2x+M) -2(M x + x^2) P + 2\lambda M (Mx + 2x^2),\\
c(x)&:= P^{-1} - P + 2\lambda (M+2x).
\end{aligned}
\end{equation}
In order to prove the lemma, it will be sufficient to show that $a(x)<0$, $b(x)<0$ and $c(x)<0$ for all $x> 0$.

We have
$$
\restr{a(x)}{M=0}= x^2\underbrace{(L^{-2x}-L^{2x})}_{<0}-4\lambda x^3 <0.
$$
Treating now $M$ as a real variable and differentiating yields:
$$
\frac{d}{dM}a(x)= -\lambda x^2 P^{-1}+ 2(x+M)\underbrace{(1-L^M)}_{<0} - \lambda (x+M)^2L^M- 2 \lambda(3x^2+2Mx)<0.
$$
This implies that $a(x)<0$ for all $x> 0$ and all $M\ge 0$. 

As far as $b(x)$ is concerned, we will decompose it further, as follows
$$
b(x)=-2x^2 P^{-1}- M(2x+M) + 2x \,\widetilde{b}(x), 
$$
where $\widetilde{b}(x):={   - (M+x)P + \lambda M(M+2x) }$.
We have $\widetilde{b}(0)=M(-L^M+\lambda M)$.
The quantity $-L^M+\lambda M$ is negative for all $M\ge 0$ because it takes the value $-1$ for $M=0$ and is a decreasing function of $M$. As a matter of fact, we have
$$
\frac{d}{dM}(-L^M+\lambda M)= -\lambda L^M + \lambda = \lambda (-L^M+1) <0.
$$
Hence $\widetilde{b}(0)<0$. We claim that $\widetilde{b}(x)$ is also decreasing in $x$, because
$$
\frac{d}{dx}\widetilde{b}(x)= -P -2\lambda (M+x)P + 2\lambda M = -P +2\lambda M (-P+1) - 2\lambda x P <0.
$$
We conclude that $\widetilde{b}(x)$ (and therefore also $b(x)$) is negative for $x\ge 0$.

Finally we show that $c(x)<0$ for $x>0$. 
We have $c(0)=L^{-M}-L^M+2\lambda M$. We claim that this quantity is non-positive for all $M\ge 0$. Indeed
$$
\restr{c(0)}{M=0}=0, \quad\text{ and }\quad \frac{d}{dM}c(0)= -\lambda L^{-M} (L^M-1)^2<0. 
$$
Moreover, since
$$
\frac{d}{dx}c(x)= -2\lambda P^{-1}-2\lambda P +4 \lambda = -2\lambda (P-1)^2<0,
$$
we conclude that also $c(x)<0$ for $x>0$. 
This implies that the function $x\mapsto E''_\ro(x)$ is  strictly decreasing in $(0,\infty)$, as claimed.
\end{proof}
In order to study the sign of $E''_\ri$ and $E''_\ro$ we proceed as follows. Notice that 
\begin{equation}\nonumber
E''_\ve (1)=2 (1-\sg)/F(1)\quad  \mbox{ and }\quad  \lim_{k\to\ali} E''_\ve (k)=-\ali \quad \mbox{for } \ve\in\{\ri,\ro\}.
 \end{equation}
Now, by applying Lemma \ref{lemma decreasing}, we get the behavior of  $E''_\ri$ and $E''_\ro$ as shown in Figure \ref{grafici}. 
\begin{figure}[h]
%\centering
\includegraphics[width=1.2\linewidth,center]{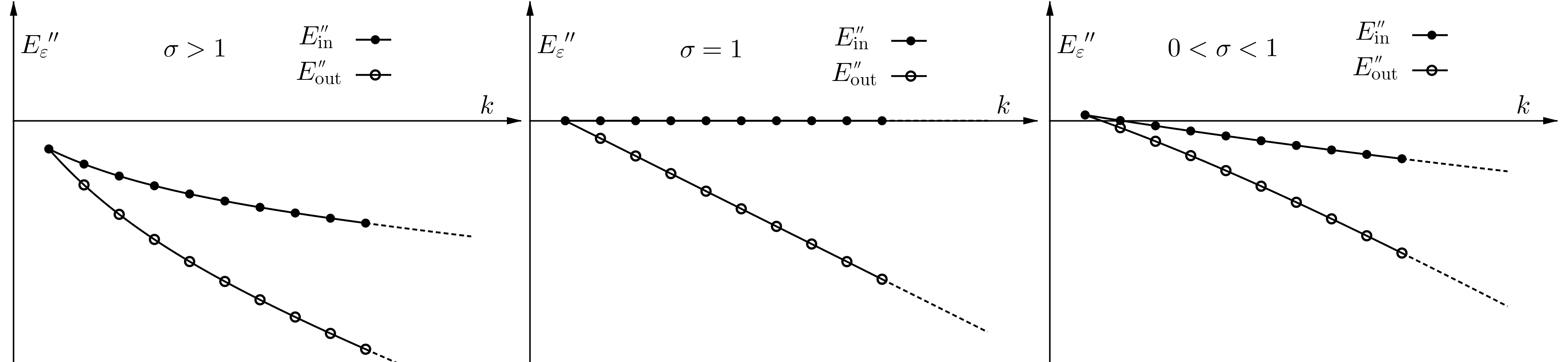}
\caption{The graphs of $E''_\ri$ and $E''_\ro$ for all possible values of $\sg$. When $\sg>1$, $E''_\ve(k)<0$ holds for all $k\ge 1$ and $\ve\in\{\ri,\ro\}$. On the other hand, when $0<\sg<1$, we have $E''_\ve(1)>0$ and $E''_\ve (k)<0$ for large $k$. The case $\sg=1$ is peculiar, as $E''_\ri\equiv 0$, while $E''_\ro(k)$ vanishes only for $k=1$ and is strictly negative otherwise. Finally it is worth noticing that the identity $E''_\ri(1)=E''_\ro(1)$ reflects the translation invariance of the functional $E$.
}  
\label{grafici}
\end{figure}

\subsection{Analysis of the resonance effects and proof of Theorem \ref{thm II}}\label{ssec 4.3} 
%$$
%%4 \, {\left(L^{m + 2 \, x} m s x + L^{m + 2 \, x} s x^{2} - L^{m + 2 \, x} m x + L^{m + 2 \, x} s x - m s x - L^{m + 2 \, x} x^{2} - s x^{2} + L^{m + 2 \, x} m + L^{m + 2 \, x} x + m x - s x + x^{2} + x\right)} {\left(L^{m + 2 \, x} s x + L^{m + 2 \, x} m + L^{m + 2 \, x} x - s x + x\right)} L^{-m - 2} {\left(s - 1\right)} s {\left(x - 1\right)}
%$$
\begin{lemma}\label{lemma res then <=0}
Suppose that $\boldsymbol{\sg>1}$. For any $k\in\{1,2,\dots\}$ and $i\in\{1,\dots, d_k\}$ that satisfy $\al_{k,i}^\ri \al_{k,i}^\ro\ne 0$, we get: 
$$
\left(\al_{k,i}^\ri\right)^2E''_\ri(k)+\left(\al_{k,i}^\ro\right)^2E''_\ro(k)+\al_{k,i}^\ri \al_{k,i}^\ro E''_{\rm res}(k) \le0, 
$$
where equality holds if and only if $k=1$ (see Figure {\rm \ref{pentagoni}}, case V).
\end{lemma}
\begin{proof}
Since, by hypothesis, $\al_{k,i}^\ro\ne0$, we can put $t:=\al_{k,i}^\ri/\al_{k,i}^\ro$. For $k$ fixed, we study the following quadratic polynomial in $t$:
$$
Q(t):=E''_\ri(k) t^2 + E''_{\rm res} (k) t +E''_\ro(k).
$$
It can be checked that the discriminant of $Q$ is 
$$
\De= \underbrace{\frac{-16 (\sg-1)(k-1) R^N}{\sg N^2 F^2}}_{\le0} \underbrace{\left( \sg k (R^{2-N-2k}-1)+ (N-2+k)R^{2-N-2k}+k  \right)}_{>0}\cdot  G,
$$
where we have set $G=(\sg-1)k (N-1+k) (R^{2-N-2k}-1)+(N-2+2k)R^{2-N-2k}$. 
We see immediately that $G>0$, as $\sg>1$ by hypothesis. We will distinguish two cases. If $k>1$, then $\De<0$ and therefore the quadratic polynomial $Q(t)$ has no real roots. Since $Q(0)=E''_\ro(k)<0$ (see Figure \ref{grafici}), then $Q$ must be strictly negative for all other values of $t$ as well. If $k=1$, then $\De=0$, which means that $Q(t)$ has one double root (which actually corresponds to $t=1$). We conclude as before. 
\end{proof}
\begin{figure}[h]\label{pentagoni}
\includegraphics[width=1.2\linewidth, center]{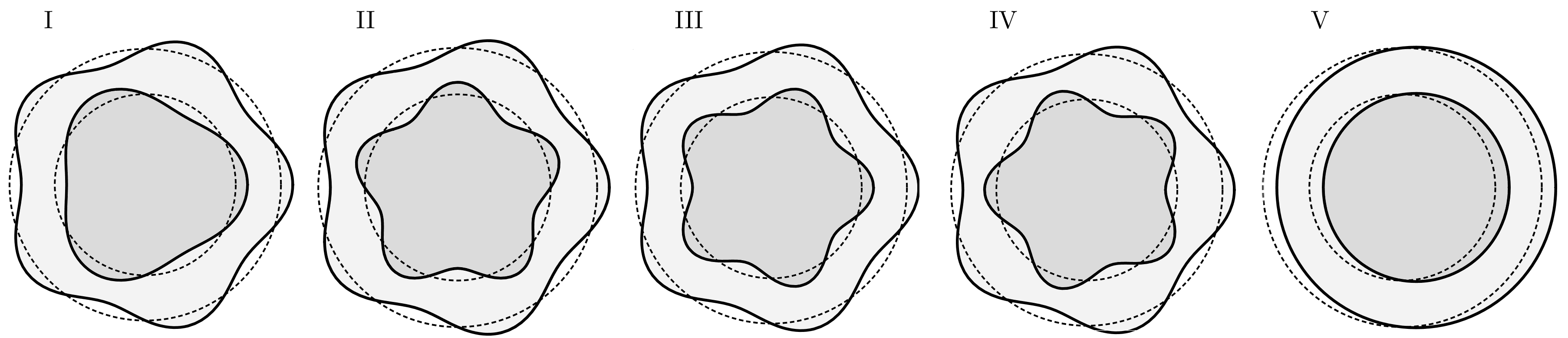}
\caption{How $(D_t,\Om_t)$ looks like for simple perturbations corresponding to
$(h_\ri\cdot n)(R\cdottone)= \al Y_{k,i}(\cdottone)$ and $h_\ro \cdot n = \be Y_{m,j}$,
for the following values of $k,i,m,j$ and $\al, \be$:\newline
I: $k=3, m=5$. II: $k=m=5$, $i\ne j$. III: $k=m=5$, $i=j$, $\al \be >0$. IV: $k=m=5$, $i=j$, $\al\be <0$. V: $k=m=1$, $i=j$, $\al=\be\ne0$.
Notice that resonance effects appear in cases III, IV and V. Moreover, as shown in Lemma \ref{lemma res then <=0}, V is the only case when $E''(\Phi)=0$ even for $\sg\ne 1$.
}
\label{pentagoni}
\end{figure}

We notice that, for all $\Phi\in\A_{\rm bar}^*$, by \eqref{barycenter preserving} and the properties of spherical harmonics, the coefficients $\al_{1,i}^\ro$ that appear in the expansion \eqref{h_in h_out exp} must vanish for $i=1,\dots,N$ (in particular we are able to avoid the case V of Figure \ref{pentagoni} by considering $\Phi\in\A_{\rm bar}^*$).
Combining this observation with the behavior of $E''_\ri$ and $E''_\ro$ shown in Figure \ref{grafici}, the result of Lemma \ref{lemma res then <=0} and Remark \ref{rmk extension} yields the main result of this paper. 
\begin{theorem}
If $\sg\ge1$, then $E''(\Phi)<0$ for all $\Phi\in\A_{\rm bar}^{*}$. In other words the configuration $(D_0,\Om_0)$ is a local maximum for the functional $E$ under the volume-preserving and barycenter-preserving constraint. 
If $\sg<1$, then there exist two perturbation fields $\Phi_1,\Phi_2\in\A_{\rm bar}^*$ such that $E''(\Phi_1)>0$ and $E''(\Phi_2)<0$. In other words, the configuration $(D_0,\Om_0)$ is a saddle shape for the functional $E$ under the volume and barycenter-preserving constraint. Notice that for $\sg=1$ we recover a local version of P\'olya's result {\rm\cite{polya}}.  
\end{theorem}

\end{document}